\documentclass[a4paper]{amsart}

\usepackage{mathtools,enumerate}
\usepackage[alphabetic]{amsrefs}
\usepackage{hyperref}

\numberwithin{equation}{section}

\newtheorem{theorem}[equation]{Theorem}

\newtheorem{proposition}[equation]{Proposition}
\newtheorem{corollary}[equation]{Corollary}
\theoremstyle{definition}
\newtheorem{definition}[equation]{Definition}
\newtheorem{example}[equation]{Example}
\newtheorem{remark}[equation]{Remark}

\DeclareMathOperator{\Aut}{Aut}
\DeclareMathOperator{\Char}{char}
\DeclareMathOperator{\End}{End}
\DeclareMathOperator{\Fix}{Fix}
\DeclareMathOperator{\Gal}{Gal}
\DeclareMathOperator{\Str}{Str}

\DeclareMathOperator{\Nr}{Nr}
\DeclareMathOperator{\Tr}{Tr}
\DeclareMathOperator{\CD}{CD}

\newcommand{\A}{\mathcal{A}}
\newcommand{\B}{\mathcal{B}}
\newcommand{\Hh}{\mathcal{H}}
\newcommand{\Ss}{\mathcal{S}}
\newcommand{\J}{\mathcal{J}}
\newcommand{\K}{\mathcal{K}}

\newcommand{\htimes}{\mathbin{\hat\times}}
\newcommand{\hsharp}{{\hat\sharp}}

\newcommand{\dash}{\nobreakdash-\hspace{0pt}}

\begin{document}

\title[Structurable algebras of skew-dimension one]{Structurable algebras of skew-dimension one and hermitian cubic norm structures}
\author{Tom De Medts}
\date{\today}

\begin{abstract}
	We study structurable algebras of skew\dash dimension one.
    We present two different equivalent constructions for such algebras:
    one in terms of non-linear isotopies of cubic norm structures, and one in terms of hermitian cubic norm structures.

    After this work was essentially finished, we became aware of the fact that
    both descriptions already occur in (somewhat hidden places in) the literature.
    Nevertheless, we prove some facts that had not been noticed before:
    \begin{enumerate}[(1)]
        \item
            We show that every form of a matrix structurable algebra can be described by our constructions;
        \item
            We give explicit formulas for the norm $\nu$;
        \item
            We make a precise connection with the Cayley--Dickson process for structurable algebras.
    \end{enumerate}
\end{abstract}

\maketitle

\section{Introduction}

\emph{Structurable algebras} form a class of non-associative algebras with involution, simultaneously generalizing Jordan algebras and associative algebras with involution.
These algebras have been introduced by Bruce Allison in 1978 \cite{Al78}, although the first non-trivial examples had already been constructed
avant la lettre by Robert Brown in 1963 \cite{Br63}; these $56$-dimensional algebras are now known as Brown algebras, and are related to groups of type $E_7$.
Any structurable algebra $\A$ gives rise to a Lie algebra $K(\A)$ via the so-called \emph{Tits-–Kantor-–Koecher construction} (TKK construction for short),
and if $\A$ is a central simple structurable $k$-algebra, then $K(\A)$ is a simple Lie algebra.
The connected component of its automorphism group,  $\Aut(K(\A))^\circ$, is an adjoint simple linear algebraic group of positive $k$-rank; see \cite[Theorem 4.1.1]{BDS17}.
It is believed that every such group arises in this fashion, but we are only aware of a proof of this fact for groups of $k$-rank $1$
(in which case it arises from a structurable \emph{division algebra}); see \cite[Theorem~4.3.1]{BDS17}.

The definition of structurable algebras makes sense only for fields $k$ of characteristic not equal to $2$ or $3$.
The central simple structurable algebras for $\Char(k) \neq 2, 3, 5$ have been classified, first in characteristic $0$ by Allison in \cite{Al78}
(who mistakenly omitted one class) and then in general by Oleg Smirnov in \cite{Sm90}.
They fall into six (non-disjoint) classes:
\begin{enumerate}[(1)]
    \item Jordan algebras (in which case the involution of $\A$ is trivial);
    \item associative algebras with involution;
    \item structurable algebras arising from hermitian forms;
    \item\label{it:one} structurable algebras of skew-dimension one [this includes the Brown algebras];
    \item tensor products of two composition algebras or forms of such algebras;
    \item a certain class of $35$-dimensional algebras related to octonions [discovered by Smirnov].
\end{enumerate}

This paper will be dealing with the structurable algebras of skew-dimension one, i.e., those of class~\eqref{it:one}.
In most places in the literature, the algebras in this class are described only as \emph{forms} of so-called
matrix structurable algebras (see Example~\ref{ex:matrix} below), i.e.,
algebras that become isomorphic to a matrix structurable algebra after a base field extension of degree at most $2$.

We will present an \emph{explicit} description of these forms.
In fact, we will give two different equivalent constructions of these algebras:
one in terms of \emph{non-linear isotopies of cubic norm structures}, and one in terms of \emph{hermitian cubic norm structures}.

\medskip

After this work was almost finished, we became aware of the fact that both descriptions already occur in the literature.
Indeed, Bruce Allison himself was aware already since the beginning of his study of structurable algebras that these forms can be described
in terms of non-linear isotopies of cubic norm structures;
this occurs in his 1979 paper \cite{Al79}, where this description occurs as example (iv) in \S 7.
His subsequent Theorem 11 shows that, in characteristic zero, every form of a matrix structurable algebra can indeed be obtained in this fashion.

Our second description, in terms of what we called hermitian cubic norm structures, is more recent, and occurs,
in a more general setting allowing infinite-dimensional algebras, in \cite[Section 7]{AFY08}.

Perhaps it is worth explaining briefly why we had previously overlooked these parts of the literature.
In the proceedings of a 1992 Oberwolfach meeting \cite{Al92}, Allison wrote:
\begin{quote}\em
    ``The description of the forms of the $2 \times 2$-matrix algebras in (c) is an open problem in general.''
\end{quote}
As we learned from Allison, this formulation was not very accurate, and hence somewhat misleading;
what he meant to say instead, is that there was no known \emph{rational construction} of these algebras,
in the sense of Seligman \cite{Se81}.
On the other hand, the more recent construction from \cite{AFY08} in terms of hermitian cubic norm structures can be viewed as a rational construction.
However, that paper deals with structurable tori, and even though the abstract of the paper mentions that
\begin{quote}\em
    ``New examples of structurable tori are obtained using a construction of structurable algebras from a semilinear version
    of cubic forms satisfying the adjoint identity,''
\end{quote}
it had gone largely unnoticed that this essentially solves the open problem that Allison alluded to in 1992.

\medskip

Nevertheless, our current paper adds some potentially useful new facts to the story.
\begin{enumerate}[(1)]
    \setlength{\itemsep}{1ex}
    \item
        We assemble these constructions in a systematic, self-contained and (hopefully) accessible way.
        The actual constructions are presented in Theorem~\ref{th:main} and Corollary~\ref{co:hcns}.
        In particular, we allow the quadratic extension $K/F$ occurring in both constructions to be split,
        and this avoids unnecessary case distinctions later; see, in particular, Example~\ref{ex:K=FF}.
    \item
        We show that our two constructions are equivalent, in the sense that they give precisely the same class of algebras
        (Theorems~\ref{th:Cstr=HC} and~\ref{th:HC=Cstr}); see also Remark~\ref{rem:same}\eqref{same:ii}.
    \item
        We show explicitly that every possible form of a matrix structurable algebra can be obtained using one of these two equivalent constructions
        (Theorem~\ref{th:classification}).
    \item
        We give explicit formulas for the norm $\nu$ (of degree $4$) of these structurable algebras (Propositions~\ref{pr:norm1} and~\ref{pr:norm2}).
    \item
        We make a precise connection with the structurable algebras arising from the Cayley--Dickson process, and we show that those correspond
        exactly to the structurable algebras arising from hermitian cubic norm structures obtained by semilinearly extending a non-unital
        (ordinary) cubic norm structure
        (Theorem~\ref{th:CD}).
        It is not unrealistic that this result could be used to answer the question whether every form of a matrix structurable algebra can
        be obtained by the Cayley--Dickson process \cite[p.\@~1267]{Al90}, and it seems more likely than not that the answer to this question is negative.
\end{enumerate}

\subsection*{Acknowledgment.}

We are grateful to Bruce Allison for some very inspiring and enlightening conversations based on an earlier version of this paper
and, not in the least, for encouraging and stimulating us to make this work accessible in a published form.

We also thank Bernhard M\"uhlherr and Richard Weiss for making their work \cite{MW17} available to us prior to publication;
their work has been the main source of inspiration for the current paper, even though this might not be directly visible in the results.

\subsection*{Assumptions on the characteristic.}

Throughout this paper, $F$ will be a commutative field with $\Char(F) \neq 2,3$ and $K$ will be a commutative associative $F$\dash algebra.
For a large part of the paper, the restriction on the characteristic is unnecessary, but our main application to structurable algebras
requires this assumption, and assuming it throughout also simplifies the rest of the exposition.
However, we will, at various places, point out where and how this assumption can be avoided.

\section{Cubic norm structures}

Various equivalent descriptions of cubic norm structures exist in the literature.
The first occurence of cubic norm structures is Kevin McCrimmon's 1969 paper \cite{McC69}; see also \cite[Part II, Chapter 4]{Taste}.
Other references are \cite[\S 38]{BoI} and the more recent \cite{GP16}.

We have chosen to take the somewhat less common definition from \cite[Chapter 15]{TW} since it is the most convenient for our purposes,
mostly because it avoids passing to scalar extensions and does not require non-degeneracy of the corresponding trace form.
See, in particular, Remark~\ref{rem:cns-char} below.
For a detailed account on why this approach is equivalent to the other, we refer to the appendix of \cite{MW17} written by Holger Petersson.

\begin{definition}\label{def:cns}
    A \emph{cubic norm structure} over $K$ is a quintuple $\J = (J, N, \sharp, T, 1)$, where $J$ is a (left) $K$-module,
    $N \colon J \to K$ is a map called the \emph{norm}, $\sharp \colon J \to J$ is a map called the \emph{adjoint},
    $T \colon J \times J \to K$ is a symmetric bilinear form called the \emph{trace}, and $1 \in J \setminus \{ 0 \}$
    is an element called the \emph{identity}, such that the following axioms hold, where
    we define $\times \colon J \times J \to J$ as
    \[ a \times b := (a+b)^\sharp - a^\sharp - b^\sharp \]
    for all $a,b \in J$:
    \begin{enumerate}[\rm (i)]
        \item $(ta)^\sharp = t^2 a^\sharp$,
        \item $N(ta) = t^3 N(a)$,
        \item $N(a + b) = N(a) + T(b, a^\sharp) + T(a, b^\sharp) + N(b)$,
        \item $a^{\sharp\sharp} = N(a)a$,
        \item $1^\sharp = 1$,
        \item $b = T(b,1).1 - 1 \times b$,
    \end{enumerate}
    for all $t \in K$ and all $a,b \in J$.

    An element $a \in J$ is called \emph{invertible} if $N(a) \neq 0$; in this case, its \emph{inverse} is defined as $a^{-1} := N(a)^{-1} a^\sharp$.

    We define the \emph{$U$-operators} on $\J$ by
    \[ U_a(b) := T(a,b)a - a^\sharp \times b \]
    for all $a,b \in J$.
    The quadratic map $U \colon J \to \End_K(J) \colon a \mapsto U_a$ makes $\J$ into a \emph{quadratic Jordan algebra}.
    In particular, the \emph{fundamental identity}
    \begin{equation}\label{eq:FI}
        U_{U_b(a)} = U_b U_a U_b
    \end{equation}
    for all $a,b \in J$ holds.
\end{definition}

\begin{remark}\label{rem:cns-char}
    When $\Char(F) = 2$ or $3$, there is a similar definition along the same lines, but more axioms are needed,
    mostly to deal with the case where $|F| \leq 3$.
    We refer to \cite[(15.15) and (15.18)]{TW} for more details.
    Notice that \cite{TW} only deals with cubic norm structures over fields, but the arguments carry over to our setting without any change.
\end{remark}

\begin{proposition}\label{pr:cns}
    Let $(J, N, \sharp, T, 1)$ be a cubic norm structure over $K$.
    Then for all $a,b,c \in J$, the following hold.
    \begin{enumerate}[\rm (i)]
        \item\label{cns:tri}
            $T(a \times b, c) = T(a \times c, b)$,
        \item\label{cns:Taash}
            $T(a, a^\sharp) = 3N(a)$,
        \item\label{cns:xx1}
            $a^\sharp \times (a \times b) = N(a) b + T(b, a^\sharp) a$,
        \item\label{cns:xx2}
            $a \times (a^\sharp \times b) = N(a) b + T(b, a) a^\sharp$,
        \item\label{cns:v}
            $N(a^\sharp) = N(a)^2$,
        \item\label{cns:U-sh}
            $U_a(b)^\sharp = U_{a^\sharp}(b^\sharp)$,
        \item\label{cns:Ua-Uash}
            $U_a U_{a^\sharp}(b) = U_{a^\sharp} U_a(b) = N(a)^2 b$,
        \item\label{cns:TU}
            $T(b, U_a(c)) = T(c, U_a(b))$,
        \item\label{cns:NU}
            $N(U_a(b)) = N(a)^2 N(b)$.
    \end{enumerate}
\end{proposition}
\begin{proof}
    For \eqref{cns:tri}, \eqref{cns:Taash} and \eqref{cns:xx1}, see the argument in \cite[(15.18)]{TW}.
    Identity \eqref{cns:xx2} follows from \eqref{cns:xx1} for all $a \in J$ for which $N(a) \neq 0$;
    the identity for all $a$ now follows by Zariski density if $K$ is a field.
    In general, however, this identity is more tricky to show; see, for instance, \cite[Appendix C.2, (2.3.17)]{Taste}.
    Identity \eqref{cns:TU} is the ``$U$ symmetry'' in \cite[C.2.2]{Taste};
    identity \eqref{cns:Ua-Uash} is \cite[C, (2.3.11)]{Taste}.
    Identities \eqref{cns:v}, \eqref{cns:U-sh} and \eqref{cns:NU} are in \cite[C.2.4]{Taste}.

\end{proof}

\begin{definition}\label{def:strgroup}
    Let $\J = (J, N, \sharp, T, 1)$ be a cubic norm structure over $K$.
    \begin{enumerate}[(i)]
        \item
            Let $z \in J$ be invertible.
            We define
            \begin{align*}
                N_z(a) &:= N(z)^{-1} N(a) , \\
                a^{\sharp_z} &:= N(z)^{-1} U_z(a^\sharp) , \\
                T_z(a, c) &:= T(a, U_z^{-1}(c)) , \\
                1_z &:= z ;
            \end{align*}
            then $\J_z := (J, N_z, \sharp_z, T_z, 1_z)$ is again a cubic norm structure over $K$,
            called the \emph{$z$\dash isotope} of $\J$.
        \item
            A map $\varphi \colon J \to J$ is \emph{semilinear} if there is some $\sigma \in \Aut(K)$ such that
            \[ \varphi(a+b) = \varphi(a) + \varphi(b)  \quad\text{and}\quad  \varphi(ta) = t^\sigma \varphi(a) \]
            for all $t \in K$ and all $a,b \in J$;
            in this case, we also say that $\varphi$ is a \emph{$\sigma$-semilinear map}.
        \item
            Let $\varphi \colon J \to J$ be a $\sigma$-semilinear bijection such that $z := \varphi(1)$ is invertible.
            Then $\varphi$ is called an \emph{autotopy} of the cubic norm structure $\J$ if
            \[ \varphi(a^\sharp) = \varphi(a)^{\sharp_z} \]
            for all $a \in J$.
            In this case, the $\sigma^{-1}$-semilinear map $\check\varphi := \varphi^{-1} U_z$ is called the \emph{adjoint} of $\varphi$;
            this terminology is explained by Proposition~\ref{pr:adj}\eqref{adj:T} below.
        \item
            The group of all autotopies of $\J$ is called the \emph{structure group} of $\J$, and is denoted by $\Str(\J)$.
            Its subgroup of all \emph{linear} autotopies will be denoted by $\Str^\circ(\J)$ and will be called the \emph{linear structure group} of $\J$.
        \item
            An autotopy $\varphi$ is called \emph{self-adjoint} if $\check\varphi = \varphi$.
            Notice that a $\sigma$-semilinear self-adjoint autotopy necessarily has $\sigma^2 = 1$.
            By the fundamental identity~\eqref{eq:FI} and Proposition~\ref{pr:cns}\eqref{cns:U-sh},
            every $U$-operator $U_b$ with $b \in J$ invertible is a self-adjoint linear autotopy.
    \end{enumerate}
\end{definition}

\begin{remark}
    What we have called the linear structure group is usually called the structure group in the literature.
    We first encountered semilinear autotopies that are not linear in the work of M\"uhlherr and Weiss on Tits polygons \cite{MW17},
    but they also occur in Allison's early work on structurable algebras \cite[p.\@~1861]{Al79} under the name
    ``(self-adjoint) $\sigma$-semisimilarity'' (with a given multiplier).
\end{remark}

\begin{proposition}\label{pr:adj}
    Let $\varphi \in \Str(\J)$ be a $\sigma$-semilinear autotopy, and let $\delta := N(\varphi(1))$.
    Then
    \begin{enumerate}[\rm (i)]
        \item\label{adj:Npsi}
            $N(\varphi(a)) = \delta N(a)^\sigma$,
        \item\label{adj:T}
            $T(\check\varphi(a), b)^\sigma = T(a, \varphi(b))$,
        \item\label{adj:psipsi}
            $\check\varphi(\varphi(a)^\sharp) = \delta^{\sigma^{-1}} a^\sharp$,
        \item\label{adj:psi-x}
            $\check\varphi(\varphi(a) \times \varphi(b)) = \delta^{\sigma^{-1}} a \times b$,
        \item\label{adj:psi-m}
            $\varphi(a^\sharp \times \check\varphi(b)) = \varphi(a)^\sharp \times b$,
        \item\label{adj:U}
            $U_{\varphi(a)} = \varphi U_a \check\varphi$,
    \end{enumerate}
    for all $a,b \in J$.
\end{proposition}
\begin{proof}
    Let $z := \varphi(1)$.
    The fact that $\varphi$ is an autotopy tell us that
    \begin{equation}\label{eq:autotopy}
        \varphi(a^\sharp) = \delta^{-1} U_z(\varphi(a)^\sharp)
    \end{equation}
    for all $a \in J$.
    If we replace $a$ by $a^\sharp$ and invoke the identity $a^{\sharp\sharp} = N(a)a$ together with the $\sigma$-semilinearity of $\varphi$, we get,
    using Proposition~\ref{pr:cns}(\ref{cns:U-sh} and \ref{cns:Ua-Uash}),
    \[ N(a)^\sigma \varphi(a) = \delta^{-1} U_z(\varphi(a^\sharp)^\sharp) = \delta^{-3} U_z U_{z^\sharp} \varphi(a)^{\sharp\sharp} = \delta^{-1} \varphi(a)^{\sharp\sharp}
        = \delta^{-1} N(\varphi(a)) \varphi(a) , \]
    from which~\eqref{adj:Npsi} follows.
    Linearizing this identity gives
    \[ T(\varphi(a), \varphi(b)^\sharp) = \delta T(a, b^\sharp)^\sigma \]
    for all $a,b \in J$.
    If we substitute $\check\varphi(a)$ for $a$ and use the fact that $\varphi \check\varphi = U_z$ by definition, we get, using Proposition~\ref{pr:cns}\eqref{cns:TU}
    together with~\eqref{eq:autotopy},
    \begin{equation}\label{eq:almost-adj3}
        T(a, \varphi(b^\sharp)) = T(\check\varphi(a), b^\sharp)^\sigma
    \end{equation}
    for all $a,b \in J$.
    We claim that this implies~\eqref{adj:T}.
    Indeed, if we linearize~\eqref{eq:almost-adj3}, we get, in particular, that
    \[ T(a, \varphi(1 \times b)) = T(\check\varphi(a), 1 \times b)^\sigma \]
    for all $a,b \in J$.
    We now invoke the fact that $b = T(b,1).1 - 1 \times b$, from which we also get
    $\varphi(b) = T(b,1)^\sigma.\varphi(1) - \varphi(1 \times b)$, and~\eqref{adj:T} follows.

    To show~\eqref{adj:psipsi}, we simply rewrite~\eqref{eq:autotopy}, using $\varphi \check\varphi = U_z$, as
    \[ a^\sharp = \delta^{-\sigma^{-1}} \check\varphi( \varphi(a)^\sharp ) \]
    for all $a \in J$. The next identity~\eqref{adj:psi-x} now follows immediately by linearizing~\eqref{adj:psipsi}.

    Next, observe that linearizing~\eqref{eq:autotopy} yields
    \begin{equation}\label{eq:autotopy-lin}
        \varphi(a \times b) = \delta^{-1} U_z(\varphi(a) \times \varphi(b))
    \end{equation}
    for all $a,b \in J$.
    If we substitute $a^\sharp$ for $a$ and $\check\varphi(b)$ for $b$ in~\eqref{eq:autotopy-lin} and use $\varphi\check\varphi = U_z$ and~\eqref{eq:autotopy}, we get
    that
    \[ \varphi(a^\sharp \times \check\varphi(b)) = \delta^{-2} U_z(U_z(\varphi(a)^\sharp) \times U_z(b)) = \varphi(a)^\sharp \times b , \]
    where the last equality follows by linearizing the identity $U_z( U_z(c)^\sharp ) = N(z)^2 c^\sharp$ in~$c$,
    which in turn follows from Proposition~\ref{pr:cns}(\ref{cns:U-sh} and \ref{cns:Ua-Uash}).
    This shows~\eqref{adj:psi-m}.
    The final identity~\eqref{adj:U} now follows by expanding both sides using the definition of the $U$-operator and invoking~\eqref{adj:T} and~\eqref{adj:psi-m}.
\end{proof}

\section{Structurable algebras defined by semilinear self-adjoint autotopies of cubic norm structures}

Recall that we assumed that $F$ is a field with $\Char(F) \neq 2,3$.
\begin{definition}
    \begin{enumerate}[(i)]
        \item
            A \emph{structurable algebra} over $F$ is a unital non-commutative non-associative algebra $\A$ with an $F$-linear involution
            \[ \overline{\phantom{x}} \colon \A \to \A \colon x \mapsto \overline{x}  \]
            such that, when we define
            \[ V_{x,y}(z) := (x\overline{y})z + (z\overline{y})x - (z\overline{x})y \]
            for all $x,y,z \in \A$, the operator identity
            \[ [V_{x,y}, V_{z,w}] = V_{V_{x,y}(z),w} - V_{z,V_{y,x}(w)} \]
            holds, for all $x,y,z,w \in \A$, where the bracket in the left-hand side is the usual Lie bracket of endomorphisms $[A,B] := AB - BA$.
        \item
            An element $x \in \A$ is called \emph{hermitian} if $\overline{x} = x$ and \emph{skew} if $\overline{x} = -x$.
            We write
            \begin{align*}
                \Hh &= \{ x \in \A \mid \overline{x} = x \} \quad \text{and} \\
                \Ss &= \{ x \in \A \mid \overline{x} = -x \}
            \end{align*}
            and observe that $\A = \Hh \oplus \Ss$ as vector spaces.
        \item
            For $x,y,z \in \A$, we write $[x,y] := xy - yx$ and $[x,y,z] := (xy)z - x(yz)$.
            For all $x,y \in \A$, we define the \emph{inner derivation}
            \[ D_{x,y} \colon \A \to \A \colon z \mapsto \tfrac{1}{3}[[x,y] + [\overline{x}, \overline{y}], z] + [z,y,x] - [z,\overline{x},\overline{y}] ; \]
            see \cite[p.\@~138]{Al78}.
        \item
            We define the \emph{center} of $\A$ as
            \[ Z(\A) := \{ x \in \Hh \mid [x, \A] = [x, \A, \A] = [\A, x, \A] = [\A, \A, x] = 0 \} . \]
            By the $F$-linearity of the multiplication, we always have $F \subseteq Z(\A)$ (where we identify $F$ with $F.1$).
            The algebra $\A$ is \emph{central} (over $F$) if $Z(\A) = F$.
        \item
            An \emph{ideal} in $\A$ is an $F$-subspace $I$ of $\A$ such that $I\A \subseteq I$, $\A I \subseteq I$ and $\overline{I} = I$.
            The algebra $\A$ is \emph{simple} if it has no proper non-trivial ideals.
    \end{enumerate}
\end{definition}

\begin{remark}
    Some authors distinguish between the algebra $\A$ itself and the pair~$(\A, \overline{\phantom{x}})$;
    they accordingly distinguish between $Z(\A)$ and $Z(\A, \overline{\phantom{x}})$ etc.
    When we write~$\A$, we always think of this algebra as being equipped with its involution.
\end{remark}

\begin{example}[{Matrix structurable algebras \cite[Example 1.9]{Al90}}]\label{ex:matrix}
    Let $\J = (J, N, \sharp, T, 1)$ be a cubic norm structure over $F$, and let $\eta \in F^\times$ be arbitrary.
    Let
    \[ \A := \begin{pmatrix} F & J \\ J & F \end{pmatrix}
        = \left\{ \begin{psmallmatrix} s_1 & j_1 \\ j_2 & s_2 \end{psmallmatrix} \bigm\vert s_1,s_2 \in F, j_1,j_2 \in J \right\} , \]
    equipped with involution
    \[ \overline{\begin{pmatrix} s_1 & j_1 \\ j_2 & s_2 \end{pmatrix}} = \begin{pmatrix} s_2 & j_1 \\ j_2 & s_1 \end{pmatrix} \]
    and with multiplication
    \[ \begin{pmatrix} s_1 & j_1 \\ j_2 & s_2 \end{pmatrix}\begin{pmatrix} s_1' & j_1' \\ j_2' & s_2' \end{pmatrix}
        = \begin{pmatrix} s_1 s_1' + \eta T(j_1, j_2') & s_1 j_1' + s_2' j_1 + \eta j_2 \times j_2' \\
            s_2 j_2' + s_1' j_2 + j_1 \times j_1' & s_2 s_2' + \eta T(j_2, j_1') \end{pmatrix} \]
    for all $s_1, s_1', s_2, s_2' \in F$ and all $j_1, j_1', j_2, j_2' \in J$.
    Then $\A$ is a structurable algebra, which is central over $F$.
    Moreover, $\A$ is simple if and only if $T$ is non-degenerate.
\end{example}

We are ready to present our first main construction of structurable algebras of skew-dimension one.
\begin{theorem}\label{th:main}
	Let $K/F$ be a quadratic \'etale extension with non-trivial Galois automorphism $\sigma$,
    let $\J = (J, N, \sharp, T, 1)$ be a cubic norm structure over $K$
    and let $\varphi \in \Str(\J)$ be a $\sigma$-semilinear self-adjoint autotopy of $\J$.
    Let $\delta := N(\varphi(1))$, and assume that there is an element
    $\gamma \in K$ such that $\gamma^{\sigma+1} = \delta^{-1}$; in particular, $\delta \in F$.

    Let $\A = K \oplus J$ as an $F$-vector space, equipped with involution
    \[ \overline{(s, b)} := (s^\sigma, b) \]
    for all $s \in K$ and $b \in J$,
    and with multiplication given by the rule
    \[ (s, b) (t, c) := \bigl( st + T(b, \varphi(c)),\; sc + t^\sigma b + \gamma \varphi(b) \times \varphi(c) \bigr) \]
    for all $s,t \in K$ and all $b,c \in J$.
    Then $\A$ is a structurable $F$-algebra of skew-dimension one.
    Moreover, if $T$ is non-degenerate, then $\A$ is a central simple structurable $F$-algebra.
\end{theorem}

\begin{proof}
    We will use the same method as in \cite[Example (v), p.\@~148]{Al78}.
    First, notice that the map $(s,b) \mapsto \overline{(s,b)} = (s^\sigma, b)$ is indeed an involution of $\A$
    because
    \begin{equation}\label{eq:Tsigma}
        T(b, \varphi(c))^\sigma = T(c, \varphi(b))
    \end{equation}
    for all $b,c \in J$; this follows from Proposition~\ref{pr:adj}\eqref{adj:T}, remembering that $\check\varphi = \varphi$ because $\varphi$ is self-adjoint.

    Using Proposition~\ref{pr:cns}\eqref{cns:tri} and Proposition~\ref{pr:adj}\eqref{adj:T},
    it is possible to verify that the $F$-bilinear form given by
    \[ \langle (s, b), (t, c) \rangle := st^\sigma + ts^\sigma + T(b, \varphi(c)) + T(c, \varphi(b)) \]
    is an \emph{invariant form} on $\A$, i.e.,
    $\langle \overline{x}, \overline{y} \rangle = \langle x, y \rangle$ and
    $\langle zx, y \rangle = \langle x, \overline{z} y \rangle$
    for all $x,y,z \in \A$.
    However, if $T$ is degenerate, then so is this form, so we cannot take the shortcut as in \emph{loc.\@~cit}.

    Instead, we have to rely on \cite[Theorem 13]{Al78}, so we have to show that $\A$ is skew-alternative, i.e. that
    \begin{equation}\label{eq:S1}
        [s, x, y] = - [x, s, y]
    \end{equation}
    for all $x,y \in \A$ and all $s \in \Ss$;
    that it satisfies
    \begin{equation}\label{eq:S2}
        [x, y, z] - [y, x, z] = [z, x, y] - [z, y, x]
    \end{equation}
    for all $x,y,z \in \Hh$;
    and that
    \begin{equation*}
        D_{x^2, x}(y) = 0
    \end{equation*}
    for all $x,y \in \Hh$, or equivalently, that
    \begin{equation}\label{eq:S3}
        \tfrac{2}{3} [x^2, x] \cdot y + \tfrac{1}{3} y \cdot [x^2, x] = (yx^2) x - (yx) x^2
    \end{equation}
    for all $x,y \in \Hh$.

    Notice that $\Ss = \{ (s, 0) \mid s \in K_0 \}$ where $K_0$ is the $F$-subspace of trace zero elements of $K$;
    this makes the verification of~\eqref{eq:S1} straightforward.
    To show~\eqref{eq:S2}, let $x,y,z \in \Hh$ and write $x = (r,a)$, $y = (s,b)$ and $z = (t,c)$ with $a,b,c \in J$ and $r,s,t \in F = \Fix_K(\sigma)$.
    Without loss of generality, we may assume that $r=s=t=0$ since elements of $F$ obviously associate with everything in $\A$.
    Define
    \[ M(a,b,c) := T(\varphi(a) \times \varphi(b), \varphi(c)) , \]
    and observe that $M$ is symmetric in the three variables by Proposition~\ref{pr:cns}\eqref{cns:tri}.
    Then using~\eqref{eq:Tsigma} and Proposition~\ref{pr:adj}\eqref{adj:psi-x}, remembering that $\check\varphi = \varphi$ and that $\delta \in F$, we get
    \begin{multline*}
    [x,y,z] = \bigl( \gamma M(a,b,c) - \gamma^\sigma M(a,b,c)^\sigma , \;  T(a, \varphi(b)) c -  T(c, \varphi(b)) a \\
        +  (a \times b) \times \varphi(c) -  \varphi(a) \times (b \times c) \bigr) ,
    \end{multline*}
    from which \eqref{eq:S2} easily follows.

    The bulk of the proof consists of showing~\eqref{eq:S3}.
    Observe that we may once again assume without loss of generality that $x = (0,a)$ and $y = (0,b)$ for some $a,b \in J$.
    (Indeed, if $z = (r,a)$ for some $r \in F$, then $D_{z^2, z} = D_{x^2, x}$.)
    We have
    \[ x^2 = \bigl(  T(a, \varphi(a)), \; 2\gamma \varphi(a)^\sharp \bigr) . \]
    Notice that $T(a, \varphi(a)) \in F$ by~\eqref{eq:Tsigma}.
    Using Proposition~\ref{pr:cns}\eqref{cns:Taash}, we compute that
    \[ [x^2, x] = (6u, 0) \quad\text{where}\quad u = \gamma N(\varphi(a)) - \gamma^\sigma N(\varphi(a))^\sigma \in K_0 , \]
    so that the left-hand side of~\eqref{eq:S3} is equal to $(0, 2 u b)$.
    On the other hand, we use Proposition~\ref{pr:adj}(\ref{adj:psipsi}, \ref{adj:psi-x} and \ref{adj:psi-m}) together with the fact that $\gamma^{\sigma+1} \delta = 1$ to get
    \begin{multline*}
        (yx^2)x = \Bigl( T(a, \varphi(a)) \, T(b, \varphi(a)) + 2 T(\varphi(b) \times a^\sharp, \varphi(a)) , \; \\
            2 \gamma^{-1} T(b, a^\sharp).a + \gamma T(a, \varphi(a)).\varphi(b) \times \varphi(a) + 2 \gamma \bigl( b \times \varphi(a)^\sharp \bigr) \times \varphi(a) \Bigr)
    \end{multline*}
    and
    \begin{multline*}
        (yx)x^2 = \Bigl( T(a, \varphi(a)) \, T(b, \varphi(a)) + 2 T(\varphi(a) \times \varphi(b), a^\sharp) , \; \\
            2 \gamma T(b, \varphi(a)).\varphi(a)^\sharp + \gamma T(a, \varphi(a)).\varphi(a) \times \varphi(b) + 2 \gamma^{-1} (a \times b) \times a^\sharp \Bigr) .
    \end{multline*}
    Therefore, in order to show~\eqref{eq:S3}, it remains to show that
    \[ ub = \gamma^{-1} T(b, a^\sharp).a + \gamma \bigl( b \times \varphi(a)^\sharp \bigr) \times \varphi(a)
        - \gamma T(b, \varphi(a)).\varphi(a)^\sharp - \gamma^{-1} (a \times b) \times a^\sharp . \]
    Using Proposition~\ref{pr:cns}(\ref{cns:xx1} and \ref{cns:xx2}), this is equivalent to showing that
    \[ ub = - \gamma^{-1} N(a) b + \gamma N(\varphi(a)) b . \]
    Since $\gamma^\sigma N(\varphi(a))^\sigma = \gamma^{-1} N(a)$ by Proposition~\ref{pr:adj}\eqref{adj:Npsi},
    this identity holds indeed, and this shows that~\eqref{eq:S3} holds.
    We conclude that $\A$ is a structurable algebra.

    Next, we observe that the commutator of two arbitrary elements $(s,b), (t,c) \in \A$ is equal to
    \[ [(s,b), (t,c)] = \bigl( T(b,\varphi(c)) - T(c,\varphi(b)), \; (s-s^\sigma)c + (t^\sigma-t)b \bigr) . \]
    We see that $(s,b)$ commutes with all elements $(t,c) \in \A$ if and only if $s \in F$ and $b = 0$.
    A fortiori, $Z(\A) = F$, so $\A$ is central.

    We finally show that $\A$ is simple if $T$ is non-degenerate.
    When $K$ is not a field, i.e., when $K = F \oplus F$, then by Example~\ref{ex:K=FF} below, $\A$ is isomorphic
    to a matrix structurable algebra as described in \cite[Example 1.9]{Al90}, and by Proposition 1.10 of \emph{loc.\@~cit.},
    $\A$ is simple in this case.
    So we may assume that $K$ is a field.

    Let $I$ be a non-trivial ideal in $\A$.
    If $I$ contains a non-zero element of the form $(s,0)$ with $s \in K$, then it contains all elements $(s,0)(t,c) = (st, sc)$ for all $t \in K$ and all $c \in J$,
    and hence $I = \A$.

    We may therefore assume that $I$ contains an element of the form $(s,b)$ with $b \neq 0$.
    If $s = 0$, then we choose an element $c \in J$ with $T(b, \varphi(c)) \neq 0$ (which exists because $T$ is non-degenerate);
    then $(0,b)(0,c)$ is an element of $I$ of the form $(r,a)$ with $r \neq 0$.
    We may therefore assume that $s \neq 0$ to begin with.
    Now choose $t \in K \setminus F$; then the commutator $[(s,b), (t,0)] = (0, (t^\sigma - t)b)$ belongs to $I$,
    hence also $(0,b) \in I$, and then $(s,0) = (s,b) - (0,b) \in I$.
    Since $s \neq 0$, we conclude by the previous paragraph that $I = \A$.
\end{proof}

\begin{remark}
    When $T$ is degenerate, the algebra $\A$ is not simple.
    Indeed, if $R \leq J$ is the radical of $T$, then $\{ (0,r) \mid r \in R \}$ is a non-trivial proper ideal of $\A$.
\end{remark}

The advantage of allowing the extension $K/F$ to be split, becomes apparent in the following example, which shows that the matrix structurable algebras
are a special case of our construction.
(This is one of the differences with the approach taken in \cite{Al79}, where the extension $K/F$ is assumed to be a field extension.)
\begin{example}\label{ex:K=FF}
    Assume that $K = F \oplus F$ with $\sigma$ the exchange involution.
    Let $\J' = (J', N, \sharp, T, 1)$ be a cubic norm structure over $F$
    and let $\J = (J, N, \sharp, T, 1)$ be the corresponding cubic norm structure over $K$
    given by
    \begin{align*}
        J &:= J' \oplus J', \\
        N(j_1, j_2) &:= \bigl( N(j_1), N(j_2) \bigr), \\
        (j_1, j_2)^\# &:= (j_1^\#, j_2^\#), \\
        T\bigl( (j_1, j_2), (j_1', j_2') \bigr) &:= \bigl( T(j_1, j_1'), T(j_2, j_2') \bigr), \\
        1_J &:= (1_{J'}, 1_{J'}),
    \end{align*}
    for all $j_1, j_2 \in J'$.
    Let $\eta \in F^\times$ be arbitrary
    and let $\varphi \colon J \to J$ be the exchange involution on $J$ multiplied by $\eta$, i.e.
    \[ \varphi(j_1, j_2) := (\eta j_2, \eta j_1) \]
    for all $j_1, j_2 \in J'$.
    Then $\delta = N(\varphi(1,1)) = N(\eta.1,\eta.1) = (\eta^3,\eta^3) = \eta^3 \in K$.
    Observe that $\varphi \in \Str(\J)$ and that $\varphi$ is $\sigma$-semilinear.

    Let $\gamma := (\eta^{-1}, \eta^{-2}) \in K$; notice that $\gamma^{\sigma+1} = \delta^{-1}$.
    We apply Theorem~\ref{th:main} with these choices, and we find $\A = K \oplus J = F \oplus F \oplus J' \oplus J'$;
    we will write the elements of $\A$ as matrices $\begin{psmallmatrix} s_1 & j_1 \\ j_2 & s_2 \end{psmallmatrix}$
    rather than $(s_1, s_2, j_1, j_2)$.
    The involution on $\A$ is now given by
    \[ \overline{\begin{pmatrix} s_1 & j_1 \\ j_2 & s_2 \end{pmatrix}} = \begin{pmatrix} s_2 & j_1 \\ j_2 & s_1 \end{pmatrix} . \]
    Now let $s = (s_1, s_2) \in K$, $t = (s_1', s_2') \in K$, $b = (j_1, j_2) \in J$ and $c = (j_1', j_2') \in J$; then we get
    \[ T(b, \varphi(c)) = \bigl( \eta T(j_1, j_2'), \eta T(j_2, j_1') \bigr) \]
    and
    \[ \gamma \varphi(b) \times \varphi(c) = (\eta^{-1}, \eta^{-2}) . (\eta j_2, \eta j_1) \times (\eta j_2', \eta j_1') = (\eta j_2 \times j_2', j_1 \times j_1') . \]
    We conclude that
    \[ \begin{pmatrix} s_1 & j_1 \\ j_2 & s_2 \end{pmatrix}\begin{pmatrix} s_1' & j_1' \\ j_2' & s_2' \end{pmatrix}
        = \begin{pmatrix} s_1 s_1' + \eta T(j_1, j_2') & s_1 j_1' + s_2' j_1 + \eta j_2 \times j_2' \\
            s_2 j_2' + s_1' j_2 + j_1 \times j_1' & s_2 s_2' + \eta T(j_2, j_1') \end{pmatrix} , \]
    and we have recovered the exact formula as in Example~\ref{ex:matrix}.
\end{example}

\begin{remark}
    If $K = F \oplus F$ and the conditions in Theorem~\ref{th:main} are satisfied with $T$ non-degenerate, then the resulting algebra is always
    isomorphic to the previous example.
    Indeed, we can choose $s_0 = \bigl( (1, -1), 0 \bigr) \in K \oplus J$; this element is skew, i.e. $\overline{s_0} = -s_0$,
    and $s_0^2 = 1$.
    It then follows from \cite[Theorem 1.13]{Al90} that $\A$ is isomorphic to a $2 \times 2$-matrix structurable algebra.

    Conversely, if $K$ is a field, we choose a skew element $s_0 \in \A$; then $s_0 = (s, 0)$ for some $s \in K$ with trace zero.
    Hence $s_0^2 = -N_{K/F}(s)$ is not a square in~$F$, and it follows again from \cite[Theorem 1.13]{Al90} that $\A$
    is not isomorphic to a $2 \times 2$-matrix structurable algebra.
\end{remark}

We also obtain a simple formula for the \emph{norm} $\nu$ of the structurable algebras from Theorem~\ref{th:main};
see \cite{AF92} for a general definition of this notion, which already appears for structurable algebras of skew-dimension one in \cite{AF84}.
Since an element $x \in \A$ is invertible if and only if $\nu(x) \neq 0$ (see \cite[Proposition 2.11]{AF84}), this gives, in particular,
a criterion for checking when our structurable algebras are division algebras.
\begin{proposition}\label{pr:norm1}
	Let $K/F$ be a quadratic \'etale extension with non-trivial Galois automorphism $\sigma$,
    and denote the norm and trace of this extension by $\Nr$ and $\Tr$, respectively.
    Let $\J$, $\varphi$, $\gamma$ and $\A = K \oplus J$ be as in Theorem~\ref{th:main}.
    Then the norm of $\A$ is given by
    \[ \nu(s,b) = \bigl( \Nr(s) - T(b, \varphi(b)) \bigr)^2 - 4 T(b^\sharp, \varphi(b)^\sharp) + 4 \Tr\bigl( \gamma^{-1} s N(b) \bigr) \]
    for all $(s,b) \in \A$.
\end{proposition}
\begin{proof}
    Let $s_0 \in \Ss$ be a fixed non-zero skew element.
    By \cite[Proposition 2.11]{AF84} (see also \cite[Proposition 5.4]{AF92}), the norm is given by
    \[ \nu(x) = \tfrac{1}{6\mu} \psi\bigl(x, U_x(s_0 x)\bigr) s_0 \]
    for all $x \in \A$, where
    \begin{align*}
        U_x(y) &:= V_{x,y}(x) = 2(x\overline{y})x - (x\overline{x})y, \\
        \psi(x,y) &:= x\overline{y} - y\overline{x} \in \Ss
    \end{align*}
    for all $x,y \in \A$.
    The formula now follows by a somewhat lengthy but straightforward computation.
\end{proof}

\section{Structurable algebras defined by hermitian cubic norm structures}

\begin{definition}\label{def:hcns}
	Let $K/F$ be a quadratic \'etale extension with non-trivial Galois automorphism $\sigma$.
    A \emph{hermitian (non-unital) cubic norm structure} over $K/F$ is a quadruple $\K = (J, N, \sharp, T)$, where $J$ is a (left) $K$-module,
    $N \colon J \to K$ is a non-zero map called the \emph{norm}, $\sharp \colon J \to J$ is a map called the \emph{adjoint},
    $T \colon J \times J \to K$ is a hermitian%
    \footnote{Recall that a map $T \colon J \times J \to K$ is called {\em hermitian} if $T(sa, tb) = s t^\sigma T(a,b)$ and $T(a,b)^\sigma = T(b,a)$
        for all $s,t \in K$ and all $a,b \in J$.}
    form called the \emph{trace}, such that the following axioms hold, where
    we define $\times \colon J \times J \to J$ as
    \[ a \times b := (a+b)^\sharp - a^\sharp - b^\sharp \]
    for all $a,b \in J$:
    \begin{enumerate}[\rm (i)]
        \item $(ta)^\sharp = t^{2\sigma} a^\sharp$, \label{hcns:i}
        \item $N(ta) = t^3 N(a)$, \label{hcns:ii}
        \item $N(a + b) = N(a) + T(b, a^\sharp) + T(a, b^\sharp) + N(b)$, \label{hcns:iii}
        \item $a^{\sharp\sharp} = N(a)a$, \label{hcns:iv}
    \end{enumerate}
    for all $t \in K$ and all $a,b \in J$.

    We define the \emph{$U$-operators} on $\K$ by
    \[ U_a(b) := T(a,b)a - a^\sharp \times b \]
    for all $a,b \in J$.
    Notice that the map $(a,b) \mapsto U_a(b)$ is quadratic in $a$ and $\sigma$\dash semilinear in $b$.
    In the same fashion as for cubic norm structures, it can be shown that the \emph{fundamental identity}
    \begin{equation}\label{eq:hFI}
        U_{U_b(a)} = U_b U_a U_b
    \end{equation}
    for all $a,b \in J$ holds.
\end{definition}

\begin{remark}
    The map $\times$ is $\sigma$-semilinear in both variables.
    In particular, we have $(ta) \times b = t^\sigma(a \times b)$ for all $t \in K$ and all $a,b \in J$, and an expression of the form $ta \times b$ is ambiguous.
    In order not to overload our notation with parentheses, we will write $t.a \times b$ for $t(a \times b)$.
\end{remark}

\begin{remark}\label{rem:hcns-char}
    When $\Char(F) = 2$ or $3$, we can give a similar definition along the same lines, but just like for (ordinary) cubic norm structures, more axioms are needed.
    The required additional axioms are exactly the same as for cubic norm structures, provided the arguments of $T$ are put in the correct order.
    As a general rule of thumb, arguments of $T$ should have ``the sharp at the right''. See also Proposition~\ref{pr:hcns} below.
    (Notice, however, that some of the other statements of Proposition~\ref{pr:hcns} take additional $\sigma$'s.)
\end{remark}

\begin{proposition}\label{pr:hcns}
    Let $\K = (J, N, \sharp, T)$ be a hermitian cubic norm structure over $K/F$.
    Then for all $a,b,c \in J$, the following hold.
    \begin{enumerate}[\rm (i)]
        \item\label{hcns:tri}
            $T(a \times b, c) = T(a \times c, b)$,
        \item\label{hcns:Taash}
            $T(a, a^\sharp) = 3N(a)$,
        \item\label{hcns:xx1}
            $a^\sharp \times (a \times b) = N(a) b + T(b, a^\sharp) a$,
        \item\label{hcns:xx2}
            $a \times (a^\sharp \times b) = N(a)^\sigma b + T(b, a) a^\sharp$,
        \item
            $N(a^\sharp) = N(a)^{2\sigma}$,
        \item\label{hcns:U-sh}
            $U_a(b)^\sharp = U_{a^\sharp}(b^\sharp)$,
        \item\label{hcns:Ua-Uash}
            $U_a U_{a^\sharp}(b) = N(a)^2 b \quad$ and $\quad U_{a^\sharp} U_a(b) = N(a)^{2\sigma} b$,
        \item\label{hcns:TU}
            $T(b, U_a(c)) = T(c, U_a(b))$,
        \item\label{hcns:NU}
            $N(U_a(b)) = N(a)^2 N(b)^\sigma$.
    \end{enumerate}
\end{proposition}
\begin{proof}
    The proofs of these facts are almost \emph{ad verbum} the same as for cubic norm structures.
    We leave the computational details to the reader.
\end{proof}

Our next goal is to show that hermitian cubic norm structures are, in fact, equivalent to cubic norm structures equipped with a semilinear
self-adjoint autotopy satisfying the conditions of Theorem~\ref{th:main}.
This will be the content of Theorems~\ref{th:Cstr=HC} and~\ref{th:HC=Cstr}.

\begin{theorem}\label{th:Cstr=HC}
	Let $K/F$ be a quadratic \'etale extension with non-trivial Galois automorphism $\sigma$,
    let $\J = (J, N, \sharp, T, 1)$ be a cubic norm structure over $K$
    and let $\varphi \in \Str(\J)$ be a $\sigma$-semilinear self-adjoint autotopy of $\J$.
    Let $\delta := N(\varphi(1))$, and assume that there is an element
    $\gamma \in K$ such that $\gamma^{\sigma+1} = \delta^{-1}$; in particular, $\delta \in F$.
    We define new maps $\hat N \colon J \to K$, $\hsharp \colon J \to J$, $\htimes \colon J \times J \to J$ and $\hat T \colon J \times J \to K$ by
    \begin{align*}
        \hat N(a) &:= \gamma^{-1} N(a) , \\
        a^\hsharp &:= \gamma \varphi(a)^\sharp , \\
        a \htimes b &:= \gamma \varphi(a) \times \varphi(b) , \\
        \hat T(a, b) &:= T(a, \varphi(b)) ,
    \end{align*}
    for all $a,b \in J$.
    Then $\K := (J, \hat N, \hsharp, \hat T)$ is a hermitian cubic norm structure over $K/F$.
\end{theorem}
\begin{proof}
    It is clear that the map $\hat T$ is hermitian because $\varphi$ is $\sigma$-semilinear.
    We now verify that the axioms \eqref{hcns:i}--\eqref{hcns:iv} from Definition~\ref{def:hcns} hold for $\K$.
    Let $t \in K$ and $a,b \in J$.
    \begin{enumerate}[(i)]\itemsep1ex
        \item
            $(ta)^\hsharp = \gamma \varphi(ta)^\sharp = \gamma (t^\sigma \varphi(a))^\sharp = \gamma t^{2\sigma} \varphi(a)^\sharp = t^{2\sigma} a^\hsharp$.
        \item
            $\hat N(ta) = \gamma^{-1} N(ta) = \gamma^{-1} t^3 N(a) = t^3 \hat N(a)$.
        \item
            By Proposition~\ref{pr:adj}\eqref{adj:psipsi} together with the fact that $\delta \in F$, we have
            \[
                \hat T( a, b^\hsharp )
                    = T\bigl( a, \varphi(\gamma \varphi(b)^\sharp) \bigr)
                    = \gamma^\sigma \delta T(a, b^\sharp)
                    = \gamma^{-1} T(a, b^\sharp) .
            \]
            Hence
            \begin{align*}
                \hat N(a+b)
                &= \gamma^{-1} N(a+b) \\
                &= \gamma^{-1} N(a) + \gamma^{-1} T(b, a^\sharp) + \gamma^{-1} T(a, b^\sharp) + \gamma^{-1} N(b) \\
                &= \hat N(a) + \hat T(b, a^\hsharp) + \hat T(a, b^\hsharp) + \hat N(b) .
            \end{align*}
        \item
            $a^{\hsharp\hsharp} = \gamma^{2\sigma+1} \varphi(\varphi(a)^\sharp)^\sharp = \gamma^{2\sigma+1} (\delta a^\sharp)^\sharp
                = \gamma^{-1} a^{\sharp\sharp} = \gamma^{-1} N(a) a = \hat N(a) a$.
        \qedhere
    \end{enumerate}
\end{proof}

We now show the converse.
\begin{theorem}\label{th:HC=Cstr}
	Let $K/F$ be a quadratic \'etale extension with non-trivial Galois automorphism $\sigma$
    and let $\K = (J, N, \sharp, T)$ be a hermitian cubic norm structure over $K/F$.
    Let $z \in J$ be an arbitrary element with $N(z) \neq 0$, and write $\gamma := N(z)^\sigma$.
    We define new maps $\hat N \colon J \to K$, $\hsharp \colon J \to J$, $\htimes \colon J \times J \to J$ and $\hat T \colon J \times J \to K$ by
    \begin{align*}
        \hat N(a) &:= \gamma N(a) , \\
        a^\hsharp &:= \gamma^{-1} U_z(a)^\sharp , \\
        a \htimes b &:= \gamma^{-1} . U_z(a) \times U_z(b) , \\
        \hat T(a, b) &:= T(a, U_z(b)) ,
    \end{align*}
    for all $a,b \in J$;
    we also define $\hat 1 := \gamma^{-1} z^\sharp \in J$.
    Then $\J := (J, \hat N, \hsharp, \hat T, \hat 1)$ is a cubic norm structure,
    and the map $\varphi := U_z^{-1}$ is a $\sigma$-semilinear self-adjoint element of the structure group $\Str(\J)$ of $\J$.
    Moreover, if we set $\delta := \hat N(\varphi(\hat 1))$, then $\gamma^{\sigma+1} = \delta^{-1}$.
\end{theorem}
\begin{proof}
    First, observe that the map $\hat T$ is indeed $K$-bilinear because $U_z$ is $\sigma$-semilinear
    and that it is symmetric by Proposition~\ref{pr:hcns}\eqref{hcns:TU}.
    We now verify that each of the axioms (i)--(vi) in Definition~\ref{def:cns} are satisfied.
    Let $t \in K$ and $a,b \in J$.
    \begin{enumerate}[(i)]\itemsep1ex
        \item
            $(ta)^\hsharp = \gamma^{-1} U_z(ta)^\sharp = \gamma^{-1} (t^\sigma U_z(a))^\sharp = \gamma^{-1}t^2 U_z(a)^\sharp = t^2 a^\hsharp$.
        \item
            $\hat N(ta) = \gamma N(ta) = \gamma t^3 N(a) = t^3 \hat N(a)$.
        \item
            By Proposition~\ref{pr:hcns}(\ref{hcns:U-sh}, \ref{hcns:TU} and~\ref{hcns:Ua-Uash}), we have
            \begin{multline*}
                T\bigl( U_z(a)^\sharp, U_z(b) \bigr)
                    = T\bigl( U_{z^\sharp}(a^\sharp), U_z(b) \bigr)
                    = T\bigl( b, U_z U_{z^\sharp}(a^\sharp) \bigr) \\
                    = T\bigl( b, N(z)^2 a^\sharp) \bigr)
                    = N(z)^{2\sigma} T(b, a^\sharp)
                    = \gamma^2 T(b, a^\sharp) .
            \end{multline*}
            Hence
            \begin{align*}
                \hat N(a+b)
                &= \gamma N(a+b) \\
                &= \gamma N(a) + \gamma T(b, a^\sharp) + \gamma T(a, b^\sharp) + \gamma N(b) \\
                &= \hat N(a) + T \bigl( \gamma^{-1} U_z(a)^\sharp, U_z(b) \bigr) + T \bigl( \gamma^{-1} U_z(b)^\sharp, U_z(a) \bigr) + \hat N(b) \\
                &= \hat N(a) + \hat T(a^\hsharp, b) + \hat T(b^\hsharp, a) + \hat N(b) .
            \end{align*}
        \item
            $a^{\hsharp\hsharp} = \gamma^{-3} U_z(U_z(a)^\sharp)^\sharp = \gamma^{-3} (U_z U_{z^\sharp}(a^\sharp))^\sharp = \gamma a^{\sharp\sharp} = \gamma N(a) a = \hat N(a) a$.
        \item
            By Proposition~\ref{pr:hcns}\eqref{hcns:Taash}, we have
            \[ U_z(z^\sharp) = T(z,z^\sharp)z - z^\sharp \times z^\sharp = 3N(z)z - 2z^{\sharp\sharp} = N(z)z = \gamma^{\sigma} z \]
            so $U_z(\hat 1) = z$ and hence
            $\hat 1^\hsharp = \gamma^{-1} U_z(\hat 1)^\sharp = \gamma^{-1} z^\sharp = \hat 1$.
        \item
            Using Proposition~\ref{pr:hcns}\eqref{hcns:xx2}, we have
            \begin{multline*}
                \hat T(b, \hat 1).\hat 1 - \hat 1 \htimes b
                = \gamma^{-1} T(b, U_z(\hat 1)) . z^\sharp - \gamma^{-1} . U_z(\hat 1) \times U_z(b) \\
                \begin{aligned}
                    &= \gamma^{-1} T(b, z) . z^\sharp - \gamma^{-1} . z \times \bigl( T(z,b)z - z^\sharp \times b \bigr) \\
                    &= \gamma^{-1} T(b, z) . z^\sharp - \gamma^{-1} T(z,b)^\sigma . z \times z + \gamma^{-1}(N(z)^\sigma b + T(b, z) z^\sharp) \bigr) \\
                    &= b .
                \end{aligned}
            \end{multline*}
    \end{enumerate}
    This shows that $\J = (J, \hat N, \hsharp, \hat T, \hat 1)$ is a cubic norm structure.

    \smallskip

    Notice that the $U$-operators $\hat U_a$ in this cubic norm structure are given by
    \begin{gather}\label{eq:hatU}
    \begin{aligned}
        \hat U_a(b)
        &= \hat T(a,b)a - a^\hsharp \htimes b \\
        &= T(a, U_z(b))a - \gamma^{-2}.(N(z)^2 a^\sharp) \times U_z(b) \\
        &= U_a U_z(b)
    \end{aligned}
    \end{gather}
    for all $a,b \in J$.
    In particular, $1 = \hat U_{\hat 1} = U_{\hat 1} U_z$ and hence $\hat 1 = z^{-1}$,
    where the inverse is taken in $\J$.

    We now verify that $\varphi = U_z^{-1} \in \Str(\J)$, or equivalently, that $U_z \in \Str(\J)$.
    Notice that $U_z(\hat 1) = z$, so we have to verify that
    \[ U_z(a^\hsharp) \mathrel{\stackrel{?}{=}} U_z(a)^{\hsharp_z} = \hat N(z)^{-1} \hat U_z(U_z(a)^\hsharp) = \gamma^{-1-\sigma} U_z^2(U_z(a)^\hsharp) \]
    for all $a \in J$.
    Applying $U_z^{-1}$, expanding $a^\hsharp$ and writing $b := U_z(a)$, we reduce this to
    \[ b^\sharp \mathrel{\stackrel{?}{=}} \gamma^{-\sigma} U_z(b^\hsharp) \]
    for all $b \in J$, and this identity holds indeed.

    To verify that $\varphi$ is self-adjoint, we have to check that $\varphi = \varphi^{-1} \hat U_{\varphi(\hat 1)}$, or equivalently, by~\eqref{eq:hatU}, that
    \[ U_z^{-1} = U_z U_{U_z^{-1}(\hat 1)} U_z . \]
    Since $U_{\hat 1} = U_z^{-1}$, this follows immediately from the fundamental identity~\eqref{eq:hFI}.

    Finally, by Proposition~\ref{pr:hcns}\eqref{hcns:NU},
    $\delta = \hat N(\varphi(\hat 1)) = \gamma N(U_z^{-1}(\hat 1)) = \gamma N(z)^{-2\sigma} N(\hat 1)^\sigma = \gamma^{-1-\sigma}$.
\end{proof}

\begin{corollary}\label{co:hcns}
	Let $K/F$ be a quadratic \'etale extension with non-trivial Galois automorphism $\sigma$
    and let $\K = (J, N, \sharp, T)$ be a hermitian cubic norm structure over $K/F$.

    Let $\A = K \oplus J$ as an $F$-vector space, equipped with involution
    \[ \overline{(s, b)} := (s^\sigma, b) \]
    for all $s \in K$ and $b \in J$,
    and with multiplication given by the rule
    \[ (s, b) (t, c) := \bigl( st + T(b, c),\; sc + t^\sigma b + b \times c \bigr) \]
    for all $s,t \in K$ and all $b,c \in J$.
    Then $\A$ is a structurable $F$-algebra of skew-dimension one.
    Moreover, if $T$ is non-degenerate, then $\A$ is a central simple structurable $F$-algebra.

    We will denote this structurable algebra by $\A(\K)$.
\end{corollary}
\begin{proof}
    This follows from Theorem~\ref{th:HC=Cstr} and Theorem~\ref{th:main}.
\end{proof}
\begin{remark}\phantomsection\label{rem:same}
    \begin{enumerate}[(i)]
        \item
            A direct proof of this result can be found in~\cite[Theorem~7.2]{AFY08}; see also \cite[equation (18) on p.\@~2290]{AFY08}.
        \item\label{same:ii}
            Combining Theorems~\ref{th:Cstr=HC} and~\ref{th:HC=Cstr} with the two constructions given in Theorem~\ref{th:main}
            and Corollary~\ref{co:hcns}, we see that the structurable algebras obtained by these two constructions
            are exactly the same.
    \end{enumerate}
\end{remark}

It is worth recording how the norm formula from Proposition~\ref{pr:norm1} translates into the setting of hermitian cubic norm structures.
\begin{proposition}\label{pr:norm2}
	Let $K/F$ be a quadratic \'etale extension with non-trivial Galois automorphism $\sigma$,
    and denote the norm and trace of this extension by $\Nr$ and $\Tr$, respectively.
    Let $\K = (J, N, \sharp, T)$ be a hermitian cubic norm structure over $K/F$
    and let $\A = K \oplus J$ be as in Corollary~\ref{co:hcns}.
    Then the norm of $\A$ is given by
    \[ \nu(s,b) = \bigl( \Nr(s) - T(b, b) \bigr)^2 - 4 T(b^\sharp, b^\sharp) + 4 \Tr\bigl( s N(b) \bigr) \]
    for all $(s,b) \in \A$.
\end{proposition}
\begin{proof}
    This follows from Proposition~\ref{pr:norm1} and the formulas in Theorem~\ref{th:Cstr=HC}.
\end{proof}

We now give a class of examples of hermitian cubic norm structures.
As we will see later (see Theorem~\ref{th:CD} below), this class corresponds exactly to the class of examples of
structurable algebras of skew-dimension one arising from the Cayley--Dickson process.
\begin{example}\label{ex:HCext}
	Let $K/F$ be a quadratic \'etale extension with non-trivial Galois automorphism $\sigma$
    and let $\J = (J, N, \sharp, T)$ be a (not necessarily unital) cubic norm structure over~$F$.
    Then we can extend $\J$ in a $\sigma$-semilinear way to a hermitian cubic norm structure $\J_K$ over~$K/F$.
    Indeed, choose an element $\lambda \in K \setminus F$, and let $J_K := J \oplus \lambda J$ with the obvious $K$-module structure.
    Then we can extend $N$, $\sharp$ and $T$ to $J_K$ by
    \begin{align*}
        N(a + \lambda b) &:= N(a) + \lambda T(b, a^\sharp) + \lambda^2 T(a, b^\sharp) + \lambda^3 N(b), \\
        (a + \lambda b)^\sharp &:= a^\sharp + \lambda^\sigma a \times b + \lambda^{2\sigma} b^\sharp, \\
        T(a + \lambda b, c + \lambda d) &:= T(a, c) + \lambda T(b, c) + \lambda^\sigma T(a, d) + \lambda\lambda^\sigma T(b, d)
    \end{align*}
    for all $a,b,c,d \in J$.
    It is easily verified that this makes $\J_K = (J_K, N, \sharp, T)$ into a hermitian cubic norm structure over $K/F$.
\end{example}

\section{The classification}\label{se:cl}

We will now show that every structurable algebra that is a form of a matrix structurable algebra can be obtained from a hermitian cubic norm structure
as in Corollary~\ref{co:hcns}
(and therefore also from the construction from Theorem~\ref{th:main}).
We closely follow the approach taken in~\cite{Sm90}, which, in turn, is based on ideas from~\cite{Al78}.

Throughout the whole section~\ref{se:cl}, we will assume that $E/F$ is a field extension of degree $\leq 2$ and that $\A$ is a structurable algebra over $F$
such that $\A_E := \A \otimes_F E$ is isomorphic to a matrix structurable algebra (see Example~\ref{ex:matrix}).
In particular, $\dim\Ss = 1$.

\begin{definition}
    Let $s_0 \in \Ss \setminus \{0\}$ be arbitrary. Since $\dim\Ss=1$, this defines $s_0$ uniquely up to a scalar.
    Following~\cite{Sm90}, we let $K := \langle \Ss \rangle_{\mathrm{alg}}$ be the $F$-subalgebra of $\A$ generated by $\Ss$,
    and we let $J := \{ h \in \Hh \mid hs + sh = 0 \text{ for all } s \in \Ss \}$.
    Notice that $J = \{ h \in \Hh \mid hs_0 + s_0h = 0 \}$.
    (In both \cite{Al78} and \cite{Sm90}, our $K$ is denoted by $\mathcal{E}$ and our $J$ is denoted by $W$.)
    We have $\overline{K} = K$; we will denote the restriction of $\overline{\phantom{x}}$ to $K$ by~$\sigma$.
\end{definition}
\begin{proposition}\label{pr:A=KJ}
    We have $\A = K \oplus J$.
    Moreover, $K/F$ is a quadratic \'etale extension, with non-trivial Galois automorphism $\sigma$.
    This extension is split if and only if $\A$ is itself isomorphic to a matrix structurable algebra.
\end{proposition}
\begin{proof}
    Let $K_E := \langle \Ss_E \rangle_{\mathrm{alg}}$ and $J_E := \{ h \in \Hh_E \mid hs_0 + s_0h = 0 \}$.
    Since $\A_E$ is a matrix structurable algebra, we know that $K_E$ is $2$-dimensional over $E$ and $\A_E = K_E \oplus J_E$.
    Since $K$ itself contains both $1$ and $s_0$, it is at least $2$-dimensional, and hence $\dim K = 2$ and $K_E = K \otimes_F E$.
    Notice that $K$ is now a unital commutative associative quadratic $F$-algebra with involution $\overline{\phantom{x}}$,
    which is therefore an \'etale extension of $F$.
    Observe that $\sigma$ is non-trivial on $K$ (because $\Ss \leq K$) and trivial on $F$, so it is indeed equal to the non-trivial element of $\Gal(K/F)$.
    The extension $K/F$ is split precisely when $s_0^2 \in F$.
    By \cite[Theorem 1.13]{Al90}, this condition holds if and only if $\A$ is a matrix structurable algebra.

    Of course, $\A_E = K_E \oplus J_E$ implies $K \cap J = 0$; it remains to show that $\A = K + J$.
    Let $x \in \A$ be arbitrary, and write $x = t + b$ with $t \in K_E$ and $b \in J_E$.
    Then $xs_0 + s_0x = 2ts_0$, which forces $t \in \A \cap K_E = K$.
    Hence also $b = x - t \in \A$, and therefore $b \in \A \cap J_E = J$.
\end{proof}

\begin{definition}
    We define maps $T \colon J \times J \to K$ and $\times \colon J \times J \to J$ by declaring
    \[ bc = T(b, c) + b \times c \]
    for all $b, c \in J$, where $bc$ denotes the multiplication in $\A$;
    since $\A = K \oplus J$, this defines $T$ and $\times$ uniquely.
\end{definition}

\begin{remark}
    In \cite{Al78} and \cite{Sm90}, the maps $T$ and $\times$ are called $h$ and $k$, respectively,
    and these maps also have the order of the arguments reversed.
    More to the point, they also define a $K$-module structure on $J$ by letting $t \circ b = t^\sigma b$ for all $t \in K$, $b \in J$,
    and they point out that, with respect to this module structure, their map $k$ is $K$-bilinear.
    This is exactly what we do \emph{not} want to do.
\end{remark}

\begin{proposition}\label{pr:ids}
    We have $K J \subseteq J$ and the multiplication $K \times J \to J$ makes $J$ into a $K$-module.
    Moreover, for all $a,b,c \in J$ and all $t \in K$, we have
    \begin{align}
        bt &= t^\sigma b, \\
        T(b, c)^\sigma &= T(c, b), \\
        T(b, tc) &= t^\sigma T(b, c), \label{ids:Tsesq} \\
        b \times c &= c \times b, \label{ids:Tsym} \\
        (tb) \times c &= t^\sigma(b \times c), \label{ids:xsemi} \\
        T(a, b \times c) &= T(b, a \times c). \label{ids:Tx}
    \end{align}
    In particular, the multiplication in $\A = K \oplus J$ satisfies
    \begin{equation}\label{eq:mult}
        (s + b)(t + c) := \bigl( st + T(b, c) \bigr) + \bigl( sc + t^\sigma b + b \times c \bigr)
    \end{equation}
    for all $s,t \in K$ and $b,c \in J$.
\end{proposition}
\begin{proof}
    The first identity is \cite[(14)]{Sm90} and the next four identities are precisely \cite[(17)]{Sm90}.
    The identity $T(a \times b, c) = T(a \times c, b)$ is \cite[(19)]{Sm90}, but notice this requires that $\times$ is not the zero map.
    This is indeed the case: if $\times$ were identically zero, then also its extension $\times_E$ to the matrix structurable algebra $\A_K$ would be identically zero,
    which is a contradiction.
    Notice that \cite{Sm90} assumes that $\Char(F) \neq 5$, but deriving these identities from the fact that $\A = K \oplus J$ (which holds by Proposition~\ref{pr:A=KJ})
    does not rely on this assumption.
\end{proof}
\begin{definition}
    We define a map $\langle \cdot,\cdot,\cdot \rangle \colon J \times J \times J \to K$ by setting
    \[ \langle a,b,c \rangle := T(a, b \times c) \]
    for all $a,b,c \in J$.
    By \eqref{ids:Tsym} and \eqref{ids:Tx}, this map is symmetric in the three variables,
    and by \eqref{ids:Tsesq} and \eqref{ids:xsemi}, this map is $K$-trilinear.
\end{definition}
We are now prepared to formulate and prove our main classification result.
\begin{theorem}\label{th:classification}
    Let $E/F$ be a field extension of degree $\leq 2$ and let $\A$ be a structurable algebra over $F$
    such that $\A_E := \A \otimes_F E$ is isomorphic to a matrix structurable algebra.
    Then there is a hermitian cubic norm structure $\K = (J, N, \sharp, T)$ such that $\A \cong \A(\K)$ (as defined in Corollary~\ref{co:hcns}).
\end{theorem}
\begin{proof}
    Let $\A = K \oplus J$ as in Proposition~\ref{pr:A=KJ}.
    By Proposition~\ref{pr:ids}, we know that $J$ is a left $K$-module
    and that $J$ comes equipped with a $\sigma$-semibilinear map $\times \colon J \times J \to J$ and a hermitian form $T \colon J \times J \to K$.
    We define a map $\sharp \colon J \to K$ by $a^\sharp := \tfrac{1}{2} a \times a$ and a map $N \colon J \to K$ by
    $N(a) = \tfrac{1}{3}T(a, a^\sharp) = \tfrac{1}{6}\langle a,a,a \rangle$ for all $a \in J$.
    In particular, the maps $\sharp$ and $N$ satisfy the identities \eqref{hcns:i} and \eqref{hcns:ii} from Definition~\ref{def:hcns}, respectively.
    It is also easy to verify that identity \eqref{hcns:iii} holds:
    by the symmetry of the trilinear map $\langle \cdot, \cdot, \cdot \rangle$, we have, for all $a,b \in J$,
    \begin{align*}
        N(a+b) &= \tfrac{1}{6} \langle a+b, a+b, a+b \rangle \\
        &= \tfrac{1}{6} \langle a,a,a \rangle + \tfrac{1}{2} \langle b,a,a \rangle + \tfrac{1}{2} \langle a,b,b \rangle + \tfrac{1}{6} \langle b,b,b \rangle \\
        &= N(a) + T(b, a^\sharp) + T(a, b^\sharp) + N(b) .
    \end{align*}
    It only remains to show identity~\eqref{hcns:iv}, i.e., we have to show that $a^{\sharp\sharp} = N(a)a$ for all $a \in J$.
    We will rely on the fact that identity~\eqref{eq:S3} holds in $\A$.
    Choose $x = (0,b)$ and $y = (0,c)$ with $b,c \in J$.
    By a computation which is very similar to what we did in the proof of Theorem~\ref{th:main},
    we deduce from~\eqref{eq:S3} that
    \begin{multline*}
        \tfrac{1}{3} \langle b,b,b \rangle^\sigma c - \tfrac{1}{3} \langle b,b,b \rangle c
            = \langle c,b,b \rangle b + 2 (b^\sharp \times c) \times b - 2 T(c,b) b^\sharp - 2 b^\sharp \times (b \times c) .
    \end{multline*}
    Substituting $tb$ for $b$ with $t \in K$, we get
    \begin{multline*}
        \tfrac{1}{3} t^{3\sigma} \langle b,b,b \rangle^\sigma c - \tfrac{1}{3} t^3 \langle b,b,b \rangle c \\
            = t^3 \langle c,b,b \rangle b + 2 t^{3\sigma} (b^\sharp \times c) \times b - 2 t^{3\sigma} T(c,b) b^\sharp - 2 t^3 b^\sharp \times (b \times c) .
    \end{multline*}
    Since we can certainly find an element $t \in K$ with $t^3 \neq t^{3\sigma}$,
    the previous two identities together imply
    \begin{align}
        \tfrac{1}{3} \langle b,b,b \rangle^\sigma c &= 2 (b^\sharp \times c) \times b - 2 T(c,b) b^\sharp, \\
        -\tfrac{1}{3} \langle b,b,b \rangle c &= \langle c,b,b \rangle b - 2 b^\sharp \times (b \times c), \label{eq:bbbc}
    \end{align}
    for all $b,c \in J$.
    We now take $b=c$ in~\eqref{eq:bbbc}, and we get
    \[ \tfrac{4}{3} \langle b,b,b \rangle b = 4 b^\sharp \times b^\sharp , \]
    hence $N(b)b = b^{\sharp\sharp}$.

    We conclude that $\K = (J, N, \sharp, T)$ is a hermitian cubic norm structure over $K/F$.
    The multiplication rule~\eqref{eq:mult} now shows that $\A$ is indeed isomorphic to the structurable algebra $\A(\K)$
    defined in Corollary~\ref{co:hcns}.
\end{proof}

\section{The Cayley--Dickson process}

In this section, we will investigate the connection between the Cayley--Dickson process for structurable algebras
(introduced by Allison and Faulkner in \cite{AF84}) and our description in terms of hermitian cubic norm structures.
As we will see, the examples arising from the Cayley--Dickson process are precisely the examples that arise
from hermitian cubic norm structures obtained from an (ordinary) cubic norm structure by semilinear extension as in Example~\ref{ex:HCext}.

\begin{definition}[{\cite[\S 6]{AF84}}]\label{def:CD}
    Let $\B$ be a Jordan algebra over $F$ with a Jordan norm $Q$ of degree $4$ and let $\mu \in F$ be a non-zero constant.
    Denote the trace of $Q$ by $t$ and define
    \[ \theta \colon \B \to \B \colon b \mapsto b^\theta := -b + \tfrac{1}{2} t(b) 1 . \]
    We define a structurable algebra $\A = \CD(\B, \mu)$ with underlying vector space $\B \oplus s_0 \B$,
    where $s_0$ is a new symbol, with involution and multiplication given by
    \[ \overline{b + s_0 c} := b - s_0 c^\theta \]
    and
    \[ (b + s_0 c)(d + s_0 e) := \bigl( bd + \mu(ce^\theta)^\theta \bigr) + s_0 \bigl( b^\theta e + (c^\theta d^\theta)^\theta \bigr) \]
    for all $b,c,d,e \in \B$.
    Then $\A$ is a simple structurable algebra of skew-dimension one; see \cite[Theorem 6.6]{AF84}.
\end{definition}

As in \cite{AF84}, we will write $\B_0 := \{ b \in \B \mid t(b) = 0 \}$.
By \cite[Theorem 5.3]{AF84}, $\B_0$ is a
non-unital cubic norm structure with
\begin{align}
    N(b) &:= \tfrac{1}{3} t(b^3), \notag \\
    b^\sharp &:= b^2 - \tfrac{1}{4} t(b^2), \label{eq:bsh} \\
    T(b,c) &:= t(bc) \notag
\end{align}
for all $b,c \in \B_0$.
(The expression $bc$ is the (commuting) product of $b$ and $c$ in the Jordan algebra $\B$.)
Observe that~\eqref{eq:bsh} linearizes to
\begin{equation}\label{eq:bxc}
    b \times c = 2bc - \tfrac{1}{2} t(bc)
\end{equation}
for all $b,c \in \B_0$.

\begin{theorem}\phantomsection\label{th:CD}
    \begin{enumerate}[\rm (i)]
        \item
            Let $\B$ be a Jordan algebra over $F$ with a Jordan norm $Q$ of degree~$4$ and let $\mu \in F$ be a non-zero constant.

            Let $\B_0$ be the non-unital cubic norm structure corresponding to~$\B$,
            let $K$ be the quadratic \'etale extension $K = F[x]/(x^2-\mu)$, and let $\K$ be the hermitian cubic norm structure obtained by
            semilinearly extending $\B_0$ over $K/F$ as in Example~\ref{ex:HCext}.
            Then $\CD(\B, \mu) \cong \A(\K)$.
        \item
            Conversely, let $\J$ be a non-unital cubic norm structure over $F$ with non-degenerate trace,
            let $K/F$ be a quadratic \'etale extension
            and let $\K$ be a hermitian cubic norm structure obtained by semilinearly extending $\J$ over $K/F$ as in Example~\ref{ex:HCext}.

            Then there is a Jordan algebra $\B$ over $F$ with a Jordan norm $Q$ of degree~$4$ and some $\mu \in F$
            such that $\A(\K) \cong \CD(\B, \mu)$.
    \end{enumerate}
\end{theorem}
\begin{proof}
    \begin{enumerate}[(i)]
        \item
            Write $K = F[s_1]$ with $s_1^2 = \mu$ and let $\K = (J, N, \sharp, T)$ with $J = \B_0 \oplus s_1 \B_0$
            and $N$, $\sharp$ and $T$ defined as in Example~\ref{ex:HCext} with $\lambda = s_1$.
            Let $\A = \A(\K)$ as in Corollary~\ref{co:hcns} and let $\A' = \CD(\B, \mu)$ as in Definition~\ref{def:CD}.
            We claim that the map
            \begin{equation}\label{eq:isom}
                \varphi \colon \A \to \A' \colon (r + s_1 t, b + s_1 c) \mapsto (r.1 + 2b) + s_0 (t.1 + 2c)
            \end{equation}
            where $r,t \in F$ and $b,c \in \B_0$,
            is an isomorphism of structurable algebras.

            Notice that the $F$-linear map $\theta$ fixes $F.1$ and inverts $\B_0$.
            In particular, $\varphi(\overline{x}) = \overline{\varphi(x)}$ for all $x \in \A$.
            Moreover, if $b,c \in \B_0$, then by~\eqref{eq:bxc},
            \begin{align*}
                bc &= \tfrac{1}{4} T(b,c) + \tfrac{1}{2} b \times c, \\
                (bc)^\theta &= \tfrac{1}{4} T(b,c) - \tfrac{1}{2} b \times c.
            \end{align*}
            So in $\A'$, we have, for all $b,c,b',c' \in \B_0$,
            \begin{multline*}
                (b + s_0 c)(b' + s_0 c') = \tfrac{1}{4} \Bigl( \bigl( T(b,b') - \mu T(c,c') \bigr) + s_0 \bigl( T(c,b') - T(b,c') \bigr) \Bigr) \\
                    + \tfrac{1}{2} \Bigl( \bigl( b \times b' + \mu c \times c' \bigr) - s_0 \bigl( b \times c' + c \times b' \bigr) \Bigr) ,
            \end{multline*}
            and this shows that the restriction of $\varphi$ to $J$ preserves the multiplication.
            The remaining verifications are straightforward.
        \item
            By \cite[Proposition 5.6]{AF84}, every non-unital cubic norm structure $\J$ with non-degenerate trace can be extended to a
            (non-simple) Jordan algebra $\B = F \oplus \J$ equipped with a Jordan norm of degree~$4$, such that $\B_0 \cong \J$.
            Let $\mu \in F$ be such that $K \cong F[x]/(x^2-\mu)$; then it follows from~(i) that $\A(\K) \cong \CD(\B, \mu)$.
        \qedhere
    \end{enumerate}
\end{proof}

\begin{remark}
    \begin{enumerate}[(i)]
        \item
            A special case of Theorem~\ref{th:CD} occurs in \cite[Lemma 7.4]{AFY08}.
        \item
            If $\mu$ is a square in $F$, then $K/F$ is split, and hence $\A(\K)$ is a matrix structurable algebra.
            In this case, we recover \cite[Proposition 6.5]{AF84},
            and it can be verified that the explicit isomorphism given in the proof of that proposition corresponds to the isomorphism
            $\varphi$ in~\eqref{eq:isom}.
    \end{enumerate}
\end{remark}

\end{document}